\documentclass[12pt,reqno]{amsart}
\usepackage[notref,notcite]{}
\usepackage{amssymb}
\usepackage{amsmath}

\numberwithin{equation}{section}

\newtheorem{theorem}{Theorem}[section]
\newtheorem{lemma}[theorem]{Lemma}

\newtheorem{corollary}[theorem]{Corollary}
\newtheorem{definition}[theorem]{Definition}

\newcommand{\cout}[1]{}

\begin{document}
\title[Killing tensor fields on the 2-torus]{Killing tensor fields on the 2-torus}

\author{Vladimir Sharafutdinov}
\address{Sobolev Institute of Mathematics and Novosibirsk State University}
\email{sharaf@math.nsc.ru}

\date{July 2014, Koltsovo}

\begin{abstract}
A symmetric tensor field on a Riemannian manifold is called Killing field if the symmetric part of its covariant derivative is equal to zero. There is a one to one correspondence between Killing tensor fields and first integrals of the geodesic flow which depend polynomially on the velocity. Therefore Killing tensor fields closely relate to the problem of integrability of geodesic flows. In particular, the following question is still open: does there exist a Riemannian metric on the 2-torus which admits an irreducible Killing tensor field of rank $\geq 3$? We obtain two necessary conditions on a Riemannian metric on the 2-torus for the existence of Killing tensor fields. The first condition is valid for Killing tensor fields of arbitrary rank and relates to closed geodesics. The second condition is obtained for rank 3 Killing tensor fields and relates to isolines of the Gaussian curvature.
\end{abstract}

\maketitle

\section{Introduction}

Although the main part of the paper concerns Killing tensor fields on the 2-torus, the problem can be posed for any Riemannian manifold. Here, we present main definitions and introduce some notations following \cite[\S3.3]{mb} as close as possible.

Given a Riemannian manifold $(M,g)$, let $\tau'_M$ be the cotangent bundle and let $S^m\tau'_M$ be the bundle of symmetric rank $m$ covariant tensors. The last notation will be mostly abbreviated to  $S^m$ on assuming the manifold to be known from the context.
The space $C^\infty(S^m)$ of smooth sections of the bundle is the $C^\infty(M)$-module of smooth covariant symmetric tensor fields of rank $m$ on $M$. The sum
$
S^*=\bigoplus_{m=0}^\infty S^m
$
is the bundle of graded commutative algebras with respect to the product $fh=\sigma(f\otimes h)$, where $\sigma$ is the symmetrization.
If $(U;x^1,\dots,x^n)$ is a local coordinate system on $M$, then the space $C^\infty(S^*;U)$ of smooth sections over $U$ is the free commutative $C^\infty(U)$-algebra with generators $dx^i\in C^\infty(\tau'_M;U)\ (1\leq i\leq n)$, i.e., every field $f\in C^\infty(S^m;U)$ can be uniquely represented in the form
$f=f_{i_1\dots i_m}\,dx^{i_1}\dots dx^{i_m}$. The coefficients of the representation $f_{i_1\dots i_m}\in C^\infty(U)$, that are called coordinates (or components) of the field $f$ (with respect to the given coordinate system), are symmetric in the indices $(i_1,\dots,i_m)$.

The differential operator
$
d=\sigma\nabla:C^\infty(S^m)\rightarrow C^\infty(S^{m+1}),
$
where $\nabla$ is the covariant derivative with respect to the Levi-Civita connection, is called the {\it inner differentiation}. We say
$f\in C^\infty(S^m)$ is a {\it Killing tensor field} if
\begin{equation} \label{df=0}
df=0.
\end{equation}
The inner derivative and product are related by the Leibnitz rule $d(fh)=(df)h+f(dh),$ which implies the statement: if $f$ and $h$ are Killing tensor fields, then $fh$ is also a Killing field. A Killing tensor field $f\in C^\infty(S^m)\ (m\neq2)$ is said to be {\it irreducible} if it cannot be represented as a finite sum $f=\sum_i u_iv_i$, where all $u_i$ and $v_i$ are Killing tensor fields of positive ranks. In the case of $m=2$, we additionally require $f$ to be different of $cg\ (c=\mbox{const})$. The requirement eliminates the metric tensor from the list of irreducible Killing fields.

Being written in coordinates for a rank $m$ tensor field, (\ref{df=0}) is a system of ${n+m}\choose {m+1}$ linear first order differential equations in ${n+m-1}\choose {m}$ coordinates of $f$, where $n=\mbox{dim}M$. Since the system is overdetermined, not every Riemannian manifold admits nonzero Killing tensor fields. In our opinion, the two-dimensional case is of the most interest since the degree of the overdetermination is equal to 1 in this case. In the two-dimensional case, roughly speaking, we obtain one equation on the metric $g$ after eliminating all coordinates of $f$ from system (\ref{df=0}), although the possibility of such elimination is rather problematic.


Let $\pi:TM\rightarrow M$ be the tangent bundle. We denote points of the manifold $TM$ by pairs $(x,\xi)$, where $x\in M$ and $\xi\in T_xM$. If $(U;x^1,\dots,x^n)$ is a local coordinate system on $M$ with the domain $U\subset M$, then the corresponding local coordinate system $(\pi^{-1}(U);x^1,\dots,x^n,\xi^1,\dots,\xi^n)$ is defined on $TM$, where $\xi=\xi^i\frac{\partial}{\partial x^i}$. Only such coordinates on $TM$ are used in what follows. Given a tensor field $f\in C^\infty(S^m)$, let $F\in C^\infty(TM)$ be defined in coordinates by $F(x,\xi)=f_{i_1\dots i_m}(x)\,\xi^{i_1}\dots\xi^{i_m}$. Observe $F(x,\xi)$ is a homogeneous polynomial of degree $m$ in $\xi$. The correspondence $f\mapsto F$ identifies the algebra $C^\infty(S^*)$ with the subalgebra of $C^\infty(TM)$ which consists of functions polynomially depending on $\xi$.

Let $H$ be the vector field on $TM$ generating the geodesic flow.
Let $\Omega M\subset TM$ be the manifold of unit tangent vectors. Since the geodesic flow preserves the norm of a vector, $H$ can be considered as a first order differential operator on $\Omega M$, i.e., $H:C^\infty(\Omega M)\rightarrow C^\infty(\Omega M)$. The operators $d$ and $H$ are related as follows: if $f\in C^\infty(S^m)$ and $F=f_{i_1\dots i_m}\,\xi^{i_1}\dots\xi^{i_m}\in C^\infty(TM)$ is the corresponding polynomial, then
$$HF=(df)_{i_1\dots i_{m+1}}\,\xi^{i_{1}}\dots\xi^{i_{m+1}}.$$ In particular, $f$ is a Killing tensor field if and only if $HF=0$, i.e., if $F$ is a first integral for the geodesic flow. Thus, the problem of finding Killing tensor fields is equivalent to the problem of finding first integrals of the geodesic flow which polynomially depend on $\xi$.

Because of the relation to integrable dynamical systems, the problem has been considered by many mathematicians, starting with classical works of G.~Darboux \cite{Dr} and J.~Birkhoff \cite{Br}, and is still investigated now. We do not present corresponding references here because of the volume limitation and refer the reader to \cite{Ko} and \cite{BF} where a large reference list is presented.
In particular, metrics on surfaces are classified which admit irreducible Killing tensor fields of rank 1 and 2. But as far as we know, the most of questions are open on metrics admitting irreducible Killing tensor fields of rank $\geq3$.

The problem is traditionally posed as follows: one has to determine whether there exist Riemannian metrics on a given manifold which admit irreducible Killing tensor fields of a given rank and, if possible, to find all such metrics. In the current paper, we discuss the more modest problem: given a Riemannian manifold $(M,g)$, one has to determine whether it admits irreducible Killing tensor fields of rank $m$ and, if possible, to describe all such fields. We are going to demonstrate right now the problem in our setting can be in principle efficiently solved if we can solve elliptic equations on the given manifold.

Unless otherwise indicated, the term ``Riemannian manifold'' means a smooth (i.e., of the class $C^\infty$) compact manifold with no boundary endowed with a smooth Riemannian metric. For a Riemannian manifold $(M,g)$, let
$
-\delta:C^\infty(S^{m+1})\rightarrow C^\infty(S^{m})
$
be the operator adjoint to $d$ with respect to the natural $L^2$ dot product defined on $C^\infty(S^m)$. The operator $\delta$ is called the {\it divergence} and is expressed by the formula
$(\delta f)_{i_1\dots i_m}=g^{pq}{\nabla}_{\!p}f_{qi_1\dots i_m}$ in local coordinates.
Since $\delta d$ is an elliptic operator \cite[Theorem 3.3.2]{mb}, it has a finite-dimensional kernel. The kernel coincides with the space of rank $m$ Killing tensor fields as is seen from the equality
$(\delta df,f)_{L^2}=-(df,df)_{L^2}$. If we found the kernel of $\delta d$ for the given $m$ and for all less ranks, then we would be able to describe efficiently all irreducible Killing tensor fields of rank $m$ on $(M,g)$.

The dimension ${n+m-1}\choose {m}$ of the bundle $S^m$ grows fast with $m$ and $n=\mbox{dim}\,M$. In Section 2, we will reduce the problem to a similar question for the elliptic operator $\delta pd$ ($p$ will be defined later) which acts on a bundle of the less dimension $\frac{n+2m-2}{n+m-2}{{n+m-2}\choose {m}}$. In particular, the latter bundle is two-dimensional in the case of $n=2$. The reduction will be done by expanding the polynomial $F\in C^\infty(\Omega M)$ into the Fourier series in spherical harmonics with respect to the variable $\xi$ and replacing the equation $HF=0$ by a chain of equations relating spherical harmonics of different degrees. The main result of Section 2 is as follows: a Killing rank $m$ tensor field $f$ is determined by its higher harmonic $pf$ uniquely up to a Killing tensor field of rank $m-2$. The spherical harmonics method is widely used for the numerical solution of the kinetic equation and of some its relatives \cite[Chapter 8]{KZ}, but sometimes the method successfully works in theoretical questions too. For example, some version of the method was used in \cite{GK} for proving the spectral rigidity of a negatively curved surface.

In Section 3, we consider Killing tensor fields on the 2D torus. Due to the existence of global isothermal coordinates, the kernel of the operator $\delta pd$ can be explicitly described, it turns out to be a two-dimensional space. Theorem \ref{Th3.2} gives some necessary and sufficient condition (of a nonlocal nature) on the metric for the existence of a rank $m$ irreducible Killing tensor field. Unfortunately, the check of the condition is not much easier than the initial problem. So, the main question remains open: does there exist a Riemannian metric on the 2-torus which admits irreducible Killing tensor fields of rank $m\geq3$? Nevertheless, the necessary condition of Corollary \ref{C3.2} allows us to give the negative answer to the question for many specific metrics. All we need is to find two closed geodesics such that certain functions $\varphi_m$ and $\psi_m$ produce linearly independent integrals over that geodesics, the functions $\varphi_m$ and $\psi_m$ are explicitly expressed through the metric and direction of a geodesic. 

System (\ref{df=0}) has the following interesting property. Each equation of the system is a linear first order differential equation in coordinates of $f$, but the system cannot be solved with respect to all first order derivatives of the coordinates. Nevertheless, as shown in \cite[Theorem 2.2.2]{mb}, after $m$-multiple differentiation, we obtain a system that can be solved with respect to all $(m+1)$ order derivatives of coordinates of $f$. In particular, a rank $m$ Killing tensor field on a connected manifold (that does not need to be compact) is uniquely determined by values of its derivatives of order $\leq m$ at one point. This allows us to estimate from above the dimension of the space of rank $m$ Killing fields by some quantity that is explicitly expressed through $m$ and $n$. We will use some version of this approach in Section 4 in studying rank 3 Killing tensor fields on the 2-torus. In this way, we obtain some new necessary condition related to the behavior of isolines of the Gaussian curvature. We observe also (although the observation is not used in the current paper) that the corresponding system for conformal Killing tensor fields possesses also a similar property \cite{DS}.

\section{The method of spherical harmonics for Killing tensor fields}

We will use the analysis of symmetric tensor fields which has been originally developed in \cite{Sh}. Then this machinery was more systematically presented in \cite{mb}; but some technical details are not included into the book although we need them here. The recent paper \cite{DS} contains \cite{Sh} as a proper subset. Therefore we will mostly refer the reader to \cite{DS} for proofs of technical statements.

By $i:S^m\rightarrow S^{m+2}$, we denote the operator of symmetric multiplication by the metric tensor, i.e., $if=fg=\sigma(f\otimes g)$. The adjoint of $i$ is the contraction $j$ with the metric tensor, it is defined in coordinates by $(jf)_{i_1\dots i_m}=g^{pq}f_{pqi_1\dots i_m}$. We will refer to $jf$ as the {\it trace} of $f$. Let $p:S^m\rightarrow S^m$ be the orthogonal projection onto the kernel of $j$.

Let ${\mbox{\rm Ker}}^mj$ be the subbundle of $S^m$ consisting of {\it trace free} tensors, i.e., of tensors $f$ satisfying $jf=0$. This terminology was introduced in \cite{DS}. We will use the terminology, although possibly ``the bundle of harmonic tensors'' is the more appropriate name for ${\mbox{\rm Ker}}^mj$. Observe that, for $f\in C^\infty({\mbox{\rm Ker}}^mj)$, the divergence
$\delta f$ is also a trace free field since the operators $j$ and $\delta$ commute \cite[Lemma 3.2]{DS}. Therefore the pair of mutually adjoint operators is defined
\begin{equation}
C^\infty({\mbox{\rm Ker}}^mj)\begin{array}{c}pd\\ \longrightarrow\\[-0.2cm] \longleftarrow\\-\delta\end{array} C^\infty({\mbox{\rm Ker}}^{m+1}j).
                         \label{2.0}
\end{equation}
The operator $\delta pd$ is elliptic as shown in \cite[Theorem 2.1]{DS}. The operators $\delta pd$ and $pd$ have coincident kernels as is seen from the equalities
$$
(\delta pdf,f)_{L^2}=-(pdf,df)_{L^2}=-(p^2df,df)_{L^2}=-(pdf,pdf)_{L^2}.
$$
Hence $pd$ has a finite-dimensional kernel.

Further formulas are a little bit different in the cases of tensor fields of even and odd rank. Therefore we will first discuss the case of an even rank and then we will present the corresponding formulas in the case of an odd rank. Both the cases can be united by a complication of notations.

By \cite[Lemma 2.3]{DS}, every field $f\in C^\infty(S^{2m})$ can be uniquely represented in the form
\begin{equation}
f=\sum\limits_{k=0}^m i^{m-k}f^k,\quad f^k\in C^\infty({\mbox{\rm Ker}}^{2k}j).
                         \label{2.1}
\end{equation}
Representation (\ref{2.1}) actually coincides with the expansion of the polynomial $F\in C^\infty(\Omega M)$ into the Fourier series in spherical harmonics with respect to the variable $\xi$. Therefore the field $f^k$ will be called the {\it harmonic of degree} $2k$ of the field $f$. In particular, $f^m=pf$ is the higher harmonic. For convenience, we assume also $f^k=0$ for $k>m$.
For an odd rank field $f\in C^\infty(S^{2m+1})$, representation (\ref{2.1}) remains true but $f^k\in C^\infty({\mbox{\rm Ker}}^{2k+1}j)$ now.

For tensor fields $f\in C^\infty(S^{2m})$ and $b\in C^\infty(S^{2m+1})$, the equation $df=b$ is equivalent to the following chain of equations \cite[Theorem 10.2]{DS} relating their harmonics:
\begin{equation}
pdf^k+\frac{2k+2}{n+4k+2}\delta f^{k+1}=b^k\quad (k=0,\dots,m),
                         \label{2.2}
\end{equation}
where $n=\mbox{dim}\,M$. For $f\in C^\infty(S^{2m+1})$ and $b\in C^\infty(S^{2m+2})$, the chain looks as follows:
\begin{equation}
\begin{aligned}
&\delta f^0=nb^0,\\
&pdf^k+\frac{2k+3}{n+4k+4}\delta f^{k+1}=b^{k+1}\quad (k=0,\dots,m).
\end{aligned}
                         \label{2.2'}
\end{equation}
Systems (\ref{2.2}) and (\ref{2.2'}) are main equations of the spherical harmonics method. These equations can be easily generalized to the case when the solution and right-hand side of the kinetic equation $HF=B$ are not polynomials but arbitrary smooth functions on $\Omega M$, as well as to the case of a more general equation containing terms responsible for the absorbtion and scattering.

Thus, $f\in C^\infty(S^{2m})$ is a Killing tensor field if and only if the following equations hold:
\begin{equation}
pdf^k+\frac{2k+2}{n+4k+2}\delta f^{k+1}=0\quad (k=0,\dots,m).
                         \label{2.3}
\end{equation}
Quite similarly, $f\in C^\infty(S^{2m+1})$ is a Killing tensor field if and only if the following equations hold:
\begin{equation}
\begin{aligned}
&\delta f^0=0,\\
&pdf^k+\frac{2k+3}{n+4k+4}\delta f^{k+1}=0\quad (k=0,\dots,m).
\end{aligned}
                         \label{2.3'}
\end{equation}

Recall \cite{mb} that $f\in C^\infty(S^{m})$ is called a {\it potential} tensor field if there exists $v\in C^\infty(S^{m-1})$ such that $f=dv$.

\begin{lemma} \label{L2.3}
If $f$ is a Killing tensor field, then the divergence $\delta f^k$ is a potential tensor field for every summand of representation (\ref{2.1}).
\end{lemma}
\begin{proof}
For definiteness, we consider a Killing field $f$ of even rank. For $\delta f^0=0$, the statement is trivial. Since $p$ coincides with the identity operator on $S^1$, equation (\ref{2.3}) for $k=0$ can be rewritten in the form $df^0+\frac{2}{n+2}\delta f^1=0$. This implies $\delta f^1$ is a potential field. Next, we continue the proof by induction in $k$. By \cite[Lemma 2.4]{DS}, the equality
$$
dpf^k=pdf^k+\frac{2k}{n+4k-2}i\delta f^k
$$
holds for any rank $2k$ tensor field $f^k$. In our case, $pf^k=f^k$ since $jf^k=0$ and the previous formula is simplified to the following one:
$$
pdf^k=df^k-\frac{2k}{n+4k-2}i\delta f^k.
$$
By the induction hypothesis, $\delta f^k=dv$ for some $v$. Substituting this expression into the previous formula and taking the permutability of $i$ and $d$ \cite[Lemma 3.2]{DS} into account, we obtain
$$
pdf^k=d\Big(f^k-\frac{2k}{n+4k-2}iv\Big).
$$
This gives together with (\ref{2.3})
$$
d\Big(f^k-\frac{2k}{n+4k-2}iv\Big)+\frac{2k+2}{n+4k+2}\delta f^{k+1}=0
$$
and we see $\delta f^{k+1}$ is a potential field.
\end{proof}

\begin{theorem} \label{Th2.2}
A rank $m$ Killing tensor field is determined by its higher harmonic uniquely up to a summand of the form $iv$ where $v$ is an arbitrary rank $m-2$ Killing tensor field. A tensor field
$f\in C^\infty({\mbox{\rm Ker}}^{m}j)$ is the higher harmonic of some Killing tensor field if and only if it satisfies the equation
\begin{equation}
pdf=0
                         \label{2.4}
\end{equation}
and has the potential divergence, i.e., $\delta f=dv$ for some $v$.
\end{theorem}

\begin{proof}
We consider the case of an even rank.
If the higher harmonic of a Killing field $f\in C^\infty(S^{2m})$ is equal to zero, then (\ref{2.1}) can be written in the form $f=i\sum_{k=0}^{m-1}i^{m-k-1}f^k$. The last equation of chain (\ref{2.3}) holds trivially and other equations of the chain mean $v=\sum_{k=0}^{m-1}i^{m-k-1}f^k$ is a Killing field.

Necessity. The potentiality of $\delta f^m$ has been proved in Lemma \ref{L2.3} and equation (\ref{2.4}) for $f^m$ coincides with (\ref{2.3}) for $k=m$.

Sufficiency. Assume $f^m\in C^\infty({\mbox{\rm Ker}}^{2m}j)$ to satisfy the equation $pdf^m=0$ and to have the potential divergence, i.e.,
\begin{equation}
dv=\delta f^m
                         \label{2.5}
\end{equation}
for some $v\in C^\infty(S^{2m-2})$. We expand $v$ into the sum of spherical harmonics
$$
v=\sum\limits_{k=0}^{m-1}i^{m-k-1}v^k,\quad jv^k=0.
$$
The corresponding sum for the field $\delta f^m$ consists of one summand. Therefore (\ref{2.5}) is equivalent to the following chain of equations:
$$
\begin{aligned}
&pdv^k+\frac{2k+2}{n+4k+2}\delta v^{k+1}=0\quad (k=0,\dots,m-2),\\
&pdv^{m-1}=\delta f^m.
\end{aligned}
$$
Setting $f^k=v^k$ for $0\leq k\leq m-2$ and $f^{m-1}=-\frac{2m}{n+4m-2}v^{m-1}$, we rewrite the system in the form
$$
pdf^k+\frac{2k+2}{n+4k+2}\delta f^{k+1}=0\quad (k=0,\dots,m-1).
$$
Together with the equation $pdf^m=0$, this gives (\ref{2.3}), i.e., $f=\sum_{k=0}^mi^{m-k}f^k$ is a Killing field.
\end{proof}

Assume we have found the kernel of the operator $pd$ from (\ref{2.0}) and let tensor fields $(f_1,\dots,f_r)$ constitute a basis of the kernel. When does the tensor field
$$
f=\alpha_1f_1+\dots+\alpha_rf_r\quad (\alpha_i\in{\mathbb C})
$$
serve as the higher harmonic of a Killing field? By Theorem \ref{Th2.2}, the potentiality of $\delta f$, i.e., the solvability of the equation
\begin{equation}
dv=\alpha_1\delta f_1+\dots+\alpha_r\delta f_r
                         \label{2.7}
\end{equation}
is the necessary and sufficient condition.
The sequence $\delta f_i\ (1\leq i\leq r)$ can be linearly dependent. We choose a maximal linearly independent subsystem of the sequence and, changing the numeration, denote the subsequence as $(\delta f_1,\dots,\delta f_s)$ with some
$s\leq r$. Then (\ref{2.7}) is replaced by the equation with less number of parameters
\begin{equation}
dv=\alpha_1\delta f_1+\dots+\alpha_s\delta f_s
                         \label{2.8}
\end{equation}
and our problem is reduced to the question: for what coefficients $(\alpha_1,\dots,\alpha_s)$ is equation (\ref{2.8}) solvable?

In view of (\ref{2.0}), the following definition is suitable: $f\in C^\infty({\mbox{\rm Ker}}^mj)$ is said to be a $j$-{\it potential} field if there exists $v\in C^\infty({\mbox{\rm Ker}}^{m-1}j)$ such that $f=pdv$. The statements ``$f$ is a potential field'' and ``$f$ is a $j$-potential field'' are not related, i.e., any of them does not imply another one for an arbitrary
$f\in C^\infty({\mbox{\rm Ker}}^mj)$. However, for a Killing field $f$, the divergence $\delta f^k$ of every harmonic is a potential and $j$-potential tensor field. The first statement is proved in Lemma \ref{L2.3} and the second statement is directly seen from (\ref{2.3}).

For the sake of completeness, we also present the following easy statement.

\begin{lemma} \label{L2.1}
Assume, for some $m\geq2$, a Riemannian manifold $(M,g)$ do not admit irreducible Killing tensor fields of ranks $1,\dots,m-1$. A rank $m$ Killing tensor field on $(M,g)$ is irreducible unless its higher harmonic is identically equal to zero.
\end{lemma}

\begin{proof}
Recall we have eliminated the metric tensor from the list of irreducible Killing fields. Every reducible Killing field can be represented as an integer coefficients polynomial of several irreducible Killing fields and of the metric tensor. This means under hypotheses of the lemma that, for an odd $m$, every reducible rank $m$ Killing tensor field is identically equal to zero; and for $m=2k$, every  reducible rank $m$ Killing tensor field is of the form $cg^k\ (c=\mbox{const})$. The higher harmonic of such a field is equal to zero.
\end{proof}

In conclusion of the section we discuss the two-dimensional case. In this case the bundle $S^m$ has dimension $m+1$ and ${\mbox{\rm Ker}}^mj$ is the two-dimensional bundle for $m>0$.

In a neighborhood of every point of a two-dimensional Riemannian manifold $(M,g)$, one can introduce {\it isothermal coordinates} $(x,y)$ such that the metric is expressed by
\begin{equation}
g=e^{2\mu(x,y)}(dx^2+dy^2)=\lambda(z)|dz|^2\quad(z=x+iy,\ \lambda(z)=e^{2\mu(x,y)}).
                            \label{2.9}
\end{equation}

Recall $\Omega M$ is the unit circle bundle. If $(x,y)$ are isothermal coordinates on $M$ and $(x,y,\xi^1,\xi^2)$ are corresponding coordinates on $TM$, then the coordinates $(x,y,\theta)$ on $\Omega M$ are defined by $\xi^1=e^{-\mu}\cos\theta, \xi^2=e^{-\mu}\sin\theta$. In these coordinates the differentiation along the geodesic flow is expressed as follows \cite[\S11]{DS}:
\begin{equation}
H=e^{-\mu}\Big(\cos\theta\frac{\partial}{\partial x}+\sin\theta\frac{\partial}{\partial y}+
(-\mu_x\sin\theta+\mu_y\cos\theta)\frac{\partial}{\partial\theta}\Big).
                            \label{H}
\end{equation}

Given a field $f\in C^\infty({\mbox{\rm Ker}}^mj)$, let us write down the equation $pdf=0$ in isothermal coordinates. The condition $jf=0$ means that
\begin{equation}
f_{\underbrace{\scriptstyle 1\dots 1}_{m-k}\underbrace{\scriptstyle 2\dots 2}_k}+f_{\underbrace{\scriptstyle 1\dots 1}_{m-k-2}\underbrace{\scriptstyle 2\dots 2}_{k+2}}=0\quad(0\leq k\leq m-2).
                            \label{2.10}
\end{equation}
The function $F\in C^\infty(\Omega M)$ corresponding to the field $f$ is obtained from the polynomial $f_{i_1\dots i_m}\xi^{i_1}\dots\xi^{i_m}$ by substituting $\xi^1=e^{-\mu}\cos\theta, \xi^2=e^{-\mu}\sin\theta$.
This gives together with (\ref{2.10})
\begin{equation}
F(x,y,\theta)=e^{-m\mu(x,y)}\big(f_{1\dots1}(x,y)\cos m\theta+f_{1\dots12}(x,y)\sin m\theta\big).
                            \label{2.11}
\end{equation}
By \cite[Lemma 5.1]{DS}, the equation $pdf=0$ is equivalent to the statement: in the Fourier series for the function $HF$, the coefficients at $\cos(m+1)\theta$ and $\sin(m+1)\theta$ are equal to zero. Using (\ref{H}) and (\ref{2.11}), we infer after easy calculations that, for a field $f\in C^\infty({\mbox{\rm Ker}}^mj)$, the equation $pdf=0$ is equivalent to the Cauchy -- Riemann system
\begin{equation}
\frac{\partial(e^{-2m\mu}f_{1\dots1})}{\partial x}-\frac{\partial(e^{-2m\mu}f_{1\dots12})}{\partial y}=0,\quad
\frac{\partial(e^{-2m\mu}f_{1\dots1})}{\partial y}+\frac{\partial(e^{-2m\mu}f_{1\dots12})}{\partial x}=0.
                            \label{2.12}
\end{equation}

\section{Killing tensor fields on the two-dimensional torus}

Recall \cite[\S6.5]{BF} there exists a global isothermal coordinate system on the two-dimensional torus ${\mathbb T}^2$ endowed with a Riemannian metric $g$. More precisely, there exists a lattice $\Gamma\subset{\mathbb R}^2={\mathbb C}$ such that ${\mathbb T}^2={\mathbb C}/\Gamma$ and the metric $g$ is expressed by (\ref{2.9}),
where $\lambda(z)=e^{2\mu(x,y)}$ is a $\Gamma$-periodic smooth function on the plane. Global isothermal coordinates on the torus are defined uniquely up to coordinate transformations of two kinds: either $z=az'+b$ or $z=a\bar z{}'+b$ with complex constants $a\neq 0$ and $b$. Transformations of the second kind can be eliminated from consideration if we fix an orientation of the torus and consider coordinate systems agreed with the orientation. Moreover, studying the invariancy of various formulas, we will restrict ourselves by considering coordinate transformations of the form $z=az'$ since the shift by a constant vector preserves tensor formulas. The group of the latter transformations coincides with the multiplicative group ${\mathbb C}\setminus\{0\}$. The commutativity of the group allows us to introduce the following definition ($\textsl{i}$ stands for the imaginary unit in what follows).

\begin{definition} \label{D3.1}
A pseudovector field $X$ of weight $m$ on a Riemannian torus $({\mathbb T}^2,g)$ is a map sending a global isothermal coordinate system to a pair  $(X^1,X^2)$ of functions on ${\mathbb T}^2$ which are transformed by the rule
\begin{equation}
(X^1+\textsl{i}X^2)(z)=a^{m}(X'{}^1+\textsl{i}X'{}^{2})(z'),\quad (X^1-\textsl{i}X^2)(z)=\bar a{}^{m}(X'{}^1-\textsl{i}X'{}^{2})(z')
                         \label{3.0}
\end{equation}
under the coordinate change $z=az'$. If the functions $(X^1,X^2)$ are constant, we speak on a constant pseudovector $X$ of weight $m$. Quite similarly, a pseudocovector field (or 1-pseudoform) $\omega$ of weight $m$ on $({\mathbb T}^2,g)$ is a map sending a global isothermal coordinate system to a pair $(\omega_1,\omega_2)$ of functions on ${\mathbb T}^2$ which are transformed by the rule
$$
\omega_1+\textsl{i}\omega_2=\bar a{}^{-m}(\omega'_1+\textsl{i}\omega'_{2}),\quad \omega_1-\textsl{i}\omega_2=a^{-m}(\omega'_1-\textsl{i}\omega'_{2}).
$$
\end{definition}

If $X$ and $\omega$ are respectively a  pseudovector and pseudocovector fields of the same weight, then $X^1\omega_1+X^2\omega_2$ is an invariant function on the torus.
Let us give also the following remark on (\ref{3.0}). If $X^1$ and $X^2$ are real functions in one global isothermal coordinate system, then the same is true in any global isothermal coordinate system; in such the case we speak on a {\it real} pseudovector field $X$ of weight $m$. For a real $X$ two equations on (\ref{3.0}) are equivalent. In the general case, the equations are independent.

\begin{theorem} \label{Th3.1}
For a two-dimensional Riemannian torus, the kernel of the elliptic operator
\begin{equation}
\delta pd:C^\infty({\mbox{\rm Ker}}^mj)\rightarrow C^\infty({\mbox{\rm Ker}}^{m}j)
                         \label{3.1}
\end{equation}
is the two-dimensional space consisting of tensor fields $f\in C^\infty({\mbox{\rm Ker}}^mj)$ whose coordinates with respect to a global isothermal coordinate system are of the form
\begin{equation}
f_{1\dots1}=e^{2m\mu}c^1,\quad f_{1\dots12}=e^{2m\mu}c^2,
                         \label{3.2}
\end{equation}
where $c=(c^1,c^2)$ is a constant pseudovector of weight $m$. The range of operator (\ref{3.1}) consists of tensor fields $f\in C^\infty({\mbox{\rm Ker}}^mj)$ whose all components with respect to a global isothermal coordinate system have zero mean values, i.e.,
\begin{equation}
\int\limits_{{\mathbb T}^2}f_{i_1\dots i_m}\,d\sigma=0,
                         \label{3.3}
\end{equation}
where $d\sigma=e^{2\mu}\,dxdy$ is the area form.
\end{theorem}

\begin{proof}
First of all we observe all coordinates of a field $f$ are determined by (\ref{3.2}) in view of (\ref{2.10}).

As we have mentioned before, the kernel of operator (\ref{3.1}) coincides with the kernel of $pd$.
If a tensor field $f\in C^\infty({\mbox{\rm Ker}}^mj)$ satisfies the equation $pdf=0$, then by (\ref{2.11}), $e^{-2m\mu}(f_{1\dots1}+\textsl{i}f_{1\dots12})$ is a holomorphic function on the torus. Hence it is a constant function, i.e., equalities (\ref{3.2}) with some complex constants $c^1$ and $c^2$ hold in any global isothermal coordinate system. These constants are real in the case of a real $f$. Starting with the rule of transforming components of a tensor field under a coordinate change, one easily checks $c=(c^1,c^2)$ is a constant pseudovector of weight $m$. The statement on the kernel of $\delta pd$ is thus proved. The statement on the range follows since $\delta pd$ is a self-adjoint operator.
\end{proof}

In the case of $m=1$, Theorem \ref{Th3.1} leads to the well known Clairuat integral for geodesics on a surface of revolution. Indeed, let $f$ be a Killing covector field on a Riemannian torus. The field $f$ belongs to the kernel of operator (\ref{3.1}) since ${\mbox{\rm Ker}}^1j=S^1$. By Theorem \ref{Th3.1}, $f_1=e^{2\mu}c^1,\ f_2=e^{2\mu}c^2$ in global isothermal coordinates for some constant vector $c$. Since $f$ is a Killing field, the function
$$
f_1\dot x+f_2\dot y=e^{2\mu}(c^1\dot x+c^2\dot y)
$$
is constant on every geodesic $\gamma(t)=\big(x(t),y(t)\big)$. Without lost of generality, we can assume $\|\dot\gamma\|^2=e^{2\mu}(\dot x{}^2+\dot y{}^2)=1$, i.e., $\dot x=e^{-\mu}\cos\varphi,\ \dot y=e^{-\mu}\sin\varphi$ where $\varphi=\varphi(t)$ is the angle between the geodesic and the coordinate line $y=\mbox{const}$. Therefore
$e^{\mu}(c^1\cos\varphi+c^2\sin\varphi)=\mbox{\rm const}$ on every geodesic. This is just the Clairuat integral in isothermal coordinates. In particular, if global isothermal coordinates are chosen so that $c=(1,0)$, the Clairuat integral takes its traditional form:
$e^{\mu}\cos\varphi=\mbox{\rm const}$.

\begin{corollary} \label{C3.1}
A tensor field $f\in C^\infty({\mbox{\rm Ker}}^mj)$ can be uniquely represented in the form
\begin{equation}
f=\tilde f+f^s,
                         \label{3.4}
\end{equation}
where $\tilde f$ belongs to the range of the operator $pd:C^\infty({\mbox{\rm Ker}}^{m-1}j)\rightarrow C^\infty({\mbox{\rm Ker}}^{m}j)$ and $f^s\in C^\infty({\mbox{\rm Ker}}^{m}j)$ has the zero divergence: $\delta f^s=0$.
\end{corollary}

The summands of representation (\ref{3.4}) will be called the $j$-{\it potential} and $j$-{\it solenoidal} parts of the field $f$ respectively.

\begin{proof}
We are looking for fields $v\in C^\infty({\mbox{\rm Ker}}^{m-1}j)$ and $f^s\in C^\infty({\mbox{\rm Ker}}^{m}j)$ satisfying the system
$$
pdv+f^s=f,\quad \delta f^s=0.
$$
Applying the operator $\delta$ to the first of these equations, we obtain
\begin{equation}
\delta pd v=\delta f.
                         \label{3.5}
\end{equation}
Conversely, if $\delta f$ belongs to the range of the operator $\delta pd$, then equation (\ref{3.5}) is solvable and we define $f^s$ by the equality $f^s=f-pdv$.

Thus, all we need is to check that, for $f\in C^\infty({\mbox{\rm Ker}}^{m}j)$, every component of the field $\delta f$ with respect to global isothermal coordinates has the zero mean value. By (\ref{2.10}), it suffices to perform the check for the components $(\delta f)_{1\dots1}$ and $(\delta f)_{1\dots12}$ only. We calculate the divergence using the relation
$f_{1\dots122}=-f_{1\dots1}$.
\begin{equation}
\begin{aligned}
(\delta f)_{1\dots1}&=g^{pq}{\nabla}_{\!p}f_{1\dots1q}=e^{-2\mu}\big({\nabla}_{\!1}f_{1\dots1}+{\nabla}_{\!2}f_{1\dots12}\big),\\
(\delta f)_{1\dots12}&=g^{pq}{\nabla}_{\!p}f_{1\dots12q}=e^{-2\mu}\big({\nabla}_{\!1}f_{1\dots12}+{\nabla}_{\!2}f_{1\dots122}\big)
=e^{-2\mu}\big(-{\nabla}_{\!2}f_{1\dots1}+{\nabla}_{\!1}f_{1\dots12}\big).
\end{aligned}
                         \label{3.6}
\end{equation}
The Christoffel symbols of the metric $e^{2\mu}(dx^2+dy^2)$ are expressed by
\begin{equation}
\begin{aligned}
&\Gamma^1_{11}=\mu_x,\quad\Gamma^1_{12}=\mu_y,\quad\Gamma^1_{22}=-\mu_x,\\
&\Gamma^2_{11}=-\mu_y,\quad\Gamma^2_{12}=\mu_x,\quad\Gamma^2_{22}=\mu_y.
\end{aligned}
                         \label{3.7}
\end{equation}
Using these equalities and the definition of the covariant derivative, we compute
$$
{\nabla}_{\!1}f_{1\dots1}=\frac{\partial f_{1\dots1}}{\partial x}-m\Gamma^p_{11}f_{1\dots1p}=\frac{\partial f_{1\dots1}}{\partial x}-m\mu_xf_{1\dots1}+m\mu_yf_{1\dots12}.
$$
Quite similarly
$$
\begin{aligned}
{\nabla}_{\!2}f_{1\dots1}&=\frac{\partial f_{1\dots1}}{\partial y}-m\mu_yf_{1\dots1}-m\mu_xf_{1\dots12},\\
{\nabla}_{\!1}f_{1\dots12}&=\frac{\partial f_{1\dots12}}{\partial x}-m\mu_yf_{1\dots1}-m\mu_xf_{1\dots12},\\
{\nabla}_{\!2}f_{1\dots12}&=\frac{\partial f_{1\dots12}}{\partial y}+m\mu_xf_{1\dots1}-m\mu_yf_{1\dots12}.
\end{aligned}
$$
Substituting these values into (\ref{3.6}), we arrive to the unexpectedly simple formulas
\begin{equation}
(\delta f)_{1\dots1}=e^{-2\mu}\Big(\frac{\partial f_{1\dots1}}{\partial x}+\frac{\partial f_{1\dots12}}{\partial y}\Big),\quad
(\delta f)_{1\dots12}=e^{-2\mu}\Big(-\frac{\partial f_{1\dots1}}{\partial y}+\frac{\partial f_{1\dots12}}{\partial x}\Big).
                         \label{3.7'}
\end{equation}
Of course, this fact should have an invariant explanation independent of coordinate calculations (I even have a guess on such an explanation but do not discuss it here). Somehow or other, these formulas imply: the expressions
$$
(\delta f)_{1\dots1}\,d\sigma=\Big(\frac{\partial f_{1\dots1}}{\partial x}+\frac{\partial f_{1\dots12}}{\partial y}\Big)dxdy,\quad
(\delta f)_{1\dots12}\,d\sigma=\Big(-\frac{\partial f_{1\dots1}}{\partial y}+\frac{\partial f_{1\dots12}}{\partial x}\Big)dxdy
$$
integrate to zero over the torus.
\end{proof}

For every integer $m\geq0$ and for every constant pseudovector $c=(c^1,c^2)$ of weight $m+1$, we introduce the tensor field $Z^{m,c}\in C^\infty({\mbox{\rm Ker}}^mj)$ on a two-dimensional Riemannian torus by setting in global isothermal coordinates
\begin{equation}
Z^{m,c}_{1\dots1}=e^{2m\mu}(c^1\mu_x+c^2\mu_y),\quad Z^{m,c}_{1\dots12}=e^{2m\mu}(c^2\mu_x-c^1\mu_y).
                         \label{3.8}
\end{equation}
Other components of the field are determined by (\ref{2.10}), where the letter $f$ should be replaced with $Z^{m,c}$. This is a correct definition, i.e., the field components are transformed in a proper way under a change of global isothermal coordinates, as one can easily check on using that $c$ is a pseudovector of weight $m+1$. In the notation $Z^{m,c}$, the first index is the rank of the field and the second index reminds the dependence  on a pseudovector $c$ of weight $m+1$. The case of $m=0$ should be specially mentioned: $Z^{0,c}=c^1\mu_x+c^2\mu_y=d\mu(c)$ is an invariant function on the torus.

\begin{theorem} \label{Th3.2}
If a Riemannian torus $({\mathbb T}^2,g)$ admits a real irreducible rank $m\geq1$ Killing tensor field, then $Z^{m-1,c}$ is a potential tensor field for some constant real pseudovector $c\neq0$ of weight $m$. Conversely, assume a Riemannian torus $({\mathbb T}^2,g)$ do not admit irreducible Killing tensor fields of ranks $1,\dots, m-1$ for some $m\geq1$  (the condition is absent in the case of $m=1$). If the tensor field $Z^{m-1,c}$ is potential for some $c\neq0$, then there exists a rank $m$ irreducible Killing tensor field on the torus.
\end{theorem}

\begin{proof}
Necessity. Let $f$ be a real irreducible rank $m$ Killing tensor field and let $pf$ be its higher harmonic. The field $pf$ is not identically equal to zero, otherwise $f$ would be reducible. By Theorem \ref{Th2.2}, $pf$ belongs to the kernel of $\delta pd$ and has a potential divergence. Applying Theorem \ref{Th3.1}, we obtain in global isothermal coordinates
\begin{equation}
(pf)_{1\dots1}=e^{2m\mu}c^1,\quad (pf)_{1\dots12}=e^{2m\mu}c^2
                         \label{3.9}
\end{equation}
 for some constant real pseudovector $c\neq0$ of weight $m$. We already calculated the divergence of an arbitrary field $pf\in C^\infty({\mbox{\rm Ker}}^mj)$ in the proof of Corollary \ref{C3.1}, namely
$$
(\delta pf)_{1\dots1}=e^{-2\mu}\Big(\frac{\partial (pf)_{1\dots1}}{\partial x}+\frac{\partial (pf)_{1\dots12}}{\partial y}\Big),\quad
(\delta pf)_{1\dots12}=e^{-2\mu}\Big(-\frac{\partial (pf)_{1\dots1}}{\partial y}+\frac{\partial (pf)_{1\dots12}}{\partial x}\Big).
$$
Substituting values (\ref{3.9}) for coordinates of $pf$, we obtain
\begin{equation}
(\delta pf)_{1\dots1}=2me^{2(m-1)\mu}(c^1\mu_x+c^2\mu_y),\quad
(\delta pf)_{1\dots12}=2me^{2(m-1)\mu}(c^2\mu_x-c^1\mu_y).
                         \label{3.9'}
\end{equation}
Comparing this with definition (\ref{3.8}) of $Z^{m,c}$, we see $Z^{m-1,c}=\frac{1}{2m}\delta(pf)$ is a potential tensor field.

Sufficiency.
Let $Z^{m-1,c}$ be a potential field for some $c\neq0$. Define the new tensor field $h\in C^\infty({\mbox{\rm Ker}}^mj)$ by setting in global isothermal coordinates
$$
h_{1\dots1}=e^{2m\mu}c^1,\quad h_{1\dots12}=e^{2m\mu}c^2.
$$
By Theorem \ref{Th3.1}, $h$ belongs to the kernel of the operator $\delta pd$. Besides this, the field has the potential divergence since $\delta h=2mZ^{m-1,c}$. Applying Theorem \ref{Th2.2}, we find a Killing tensor field $f$ whose higher harmonic is $h$. By Lemma \ref{L2.1}, $f$ is an irreducible Killing field.
\end{proof}

Theorem \ref{Th3.2} actually reduces the problem of finding rank $m$ Killing tensor fields on a two-dimensional Riemannian torus to the following question: for which constant pseudovectors $c$ of weight $m-1$ is the equation
\begin{equation}
dv=Z^{m-1,c}
                         \label{3.10'}
\end{equation}
solvable? At first sight, this equation is not easier than the initial equation (\ref{df=0}). However, let us observe that the order of equation (\ref{3.10'}) is much less than the order of (\ref{df=0}).
Indeed, being written in coordinates, (\ref{df=0}) is a system of ${n+m}\choose {m+1}$ linear first order differential equations in coordinates of the field $f$. But the corresponding system for (\ref{3.10'}) consists of ${n+m-2}\choose {m-1}$ equations, although these are inhomogeneous equations.

In the case of $m=1$, the potentiality of $Z^{0,c}$ means that $Z^{0,c}=c^1\mu_x+c^2\mu_y=0$ (a rank 0 potential tensor field is the function identically equal to zero). By an appropriate change of global isothermal coordinates, we can achieve $c=(1,0)$ and the previous equation becomes: $\mu_x=0$. Thus, in the case of $m=1$, Theorem \ref{Th3.2} is equivalent to the classical result: if a two-dimensional Riemannian torus admits a nontrivial Killing vector field, then $\mu=\mu(y)$ in some global isothermal coordinate system.

In the case of $m=2$, the potentiality of $Z^{1,c}$ means the existence of a function $v\in C^\infty({\mathbb T}^2)$ such that
$$
v_x=e^{2\mu}(c^1\mu_x+c^2\mu_y),\quad v_y=e^{2\mu}(c^2\mu_x-c^1\mu_y).
$$
We again change isothermal coordinates so that $c=(1,0)$ and obtain
$$
v_x=e^{2\mu}\mu_x,\quad v_y=-e^{2\mu}\mu_y.
$$
Eliminating the function $v$ from the system, we arrive to the equation $\partial^2e^{2\mu}/\partial x\partial y=0$.
 Thus, in the case of $m=2$, Theorem \ref{Th3.2} is equivalent to the classical result: if a Riemannian torus $({\mathbb T}^2,g)$ admits a rank 2 irreducible Killing tensor field, then the metric has the form $g=(a(x)+b(y))(dx^2+dy^2)$ in an appropriate global isothermal coordinate system.

\begin{corollary} \label{C3.2}
If a Riemannian torus $({\mathbb T}^2,g)$ admits a real irreducible rank $m+1\geq1$ Killing tensor field, then for some real constant pseudovector $c\neq0$ of weight $m+1$, the equality
\begin{equation}
\oint\limits_\gamma e^{m\mu}\big((c^1\mu_x+c^2\mu_y)\cos m\varphi+(c^2\mu_x-c^1\mu_y)\sin m\varphi\big)\,dt=0
                         \label{3.10}
\end{equation}
holds for every closed geodesic $\gamma$, where $\varphi=\varphi(t)$ is the angle between the geodesic and the coordinate line $y=\mbox{\rm const}$ of a global isothermal coordinate system.
\end{corollary}

\begin{proof}
As known \cite{mb}, every potential tensor field belongs to the kernel of the ray transform, i.e., integrates to zero over every closed geodesic. In our case, $Z^{m,c}$ is a potential field by Theorem \ref{Th3.2} and hence
\begin{equation}
\oint\limits_\gamma Z^{m,c}_{i_1\dots i_m}\dot\gamma{}^{i_1}\dots\dot\gamma{}^{i_m}\,dt=0
                         \label{3.11}
\end{equation}
for every closed geodesic $\gamma$. If $\gamma(t)=\big(x(t),y(t)\big)$, then
$$
Z^{m,c}_{i_1\dots i_m}\dot\gamma{}^{i_1}\dots\dot\gamma{}^{i_m}=\sum\limits_{k\geq0}{m\choose k}Z^{m,c}_{1\dots1{\underbrace{\scriptstyle 2\dots2}_k}}\dot x{}^{m-k}\dot y{}^{k}.
$$
We assume the binomial coefficient ${m\choose k}=\frac{m!}{k!(m-k)!}$ to be equal to zero for $k>m$, this allows us do not designate the upper summation limit on the right-hand side. Separating summands corresponding to even and odd $k$, we get
$$
Z^{m,c}_{i_1\dots i_m}\dot\gamma{}^{i_1}\dots\dot\gamma{}^{i_m}=\sum\limits_{k\geq0}{m\choose {2k}}Z^{m,c}_{1\dots1{\underbrace{\scriptstyle 2\dots2}_{2k}}}\dot x{}^{m-2k}\dot y{}^{2k}+
\sum\limits_{k\geq0}{m\choose {2k+1}}Z^{m,c}_{1\dots1{\underbrace{\scriptstyle 2\dots2}_{2k+1}}}\dot x{}^{m-2k-1}\dot y{}^{2k+1}.
$$
In view of  $jZ^{m,c}=0$, the equalities
$$
Z^{m,c}_{1\dots1{\underbrace{\scriptstyle 2\dots2}_{2k}}}=(-1)^kZ^{m,c}_{1\dots1},\quad Z^{m,c}_{1\dots1{\underbrace{\scriptstyle 2\dots2}_{2k+1}}}=(-1)^kZ^{m,c}_{1\dots12}
$$
hold. We use them to transform the previous formula to the form
$$
Z^{m,c}_{i_1\dots i_m}\dot\gamma{}^{i_1}\dots\dot\gamma{}^{i_m}=Z^{m,c}_{1\dots1}\sum\limits_{k\geq0}(-1)^k{m\choose {2k}}\dot x{}^{m-2k}\dot y{}^{2k}+
Z^{m,c}_{1\dots12}\sum\limits_{k\geq0}(-1)^k{m\choose {2k+1}}\dot x{}^{m-2k-1}\dot y{}^{2k+1}.
$$
Without lost of generality, we can assume $\|\dot\gamma\|^2=e^{2\mu}(\dot x{}^2+\dot y{}^2)=1$, i.e., that $\dot x=e^{-\mu}\cos\varphi,\ \dot y=e^{-\mu}\sin\varphi$ and the last formula becomes
$$
Z^{m,c}_{i_1\dots i_m}\dot\gamma{}^{i_1}\dots\dot\gamma{}^{i_m}=e^{-m\mu}\Big(Z^{m,c}_{1\dots1}\cos m\varphi+Z^{m,c}_{1\dots12}\sin m\varphi\Big).
$$
Substituting values (\ref{3.8}) of components of $Z^{m,c}$, we obtain
$$
Z^{m,c}_{i_1\dots i_m}\dot\gamma{}^{i_1}\dots\dot\gamma{}^{i_m}=e^{m\mu}\Big((c^1\mu_x+c^2\mu_y)\cos m\varphi+(c^2\mu_x-c^1\mu_y)\sin m\varphi\Big).
$$
Finally, inserting this expression into (\ref{3.11}), we arrive to (\ref{3.10}).
\end{proof}

After rewriting (\ref{3.10}) in the form
$$
c^1\oint\limits_\gamma e^{m\mu}(\mu_x\cos m\varphi-\mu_y\sin m\varphi)\,dt+c^2\oint\limits_\gamma e^{m\mu}(\mu_y\cos m\varphi+\mu_x\sin m\varphi)\,dt=0,
$$
we see the ratio
\begin{equation}
\frac{\oint_\gamma e^{m\mu}(\mu_x\cos m\varphi-\mu_y\sin m\varphi)\,dt}{\oint_\gamma e^{m\mu}(\mu_y\cos m\varphi+\mu_x\sin m\varphi)\,dt}
                         \label{3.12}
\end{equation}
is independent of $\gamma$. Thus, if we succeeded in finding two closed geodesics such that the ratio (\ref{3.12}) took different values for them, then the Riemannian torus would not admit a rank $m+1$ irreducible Killing tensor field.

\section{A rank 3 Killing tensor field on the two-dimensional torus}

Let $({\mathbb T}^2,g)=\big({\mathbb R}^2/\Gamma,e^{2\mu}(dx^2+dy^2)\big)$ be a Riemannian torus. Given a constant pseudovector $c=(c^1,c^2)$ of weight 3, we introduce the tensor field
$T^c\in C^\infty(\mbox{Ker}^2j)$ by setting in global isothermal coordinates
\begin{equation}
T^c_{11}=-T^c_{22}=e^{4\mu}(-c^2\mu_x+c^1\mu_y),\quad T^c_{12}=e^{4\mu}(c^1\mu_x+c^2\mu_y).
                         \label{4.0}
\end{equation}
Up to notations, this coincides with (\ref{3.8}) in the case of $m=2$. Indeed,
$$
T^c=Z^{2,c^\bot},\quad\mbox{where}\quad c^\bot=(-c^2,c^1).
$$
If $c$ is a constant pseudovector of weight 3, then $c^\bot$ is also a constant pseudovector of weight 3.

Let $f\in C^\infty(S^3)$ be a rank 3 real irreducible tensor field on the torus $({\mathbb T}^2,g)$. Equations (\ref{2.3'}) look as follows in this case:
\begin{equation}
\delta f^0=0,\quad pdf^0+\frac{1}{2}\delta f^1=0,\quad pdf^1=0.
                         \label{4.1}
\end{equation}
As we know (Theorem \ref{Th3.1}), the last equation of the system means that in global isothermal coordinates
$$
f^1_{111}=-f^1_{122}=\frac{1}{3}c^1e^{6\mu},\quad f^1_{112}=-f^1_{222}=\frac{1}{3}c^2e^{6\mu}
$$
for some real constant pseudovector $0\neq c=(c^1,c^2)$ of weight 3. The coefficient $1/3$ is included here to simplify further formulas. We have already calculated the divergence of this tensor field (formulas (\ref{3.9'})): $\delta f^1=2Z^{2,c}$. In this way system (\ref{4.1}) is reduced to the following one:
\begin{equation}
\delta f^0=0,\quad pdf^0=-Z^{2,c}.
                         \label{4.2}
\end{equation}

Now, we calculate the divergence of the covector field $f^0$ by standard rules
$$
\delta f^0=g^{pq}{\nabla}_{\!p}f^0_q=e^{-2\mu}\Big(\frac{\partial f^0_1}{\partial x}+\frac{\partial f^0_2}{\partial y}\Big).
$$
Therefore the first equation of system (\ref{4.2}) looks in global isothermal coordinates as follows:
$$
\frac{\partial f^0_1}{\partial x}+\frac{\partial f^0_2}{\partial y}=0.
$$
We satisfy this equation by setting
\begin{equation}
f^0_1=-{\nabla}_{\!2}u=-u_y,\quad f^0_2={\nabla}_{\!1}u=u_x,
                         \label{4.3}
\end{equation}
where $u(x,y)$ is a smooth real function on the plane whose partial derivatives $u_x$ and $u_y$ are $\Gamma$-periodic. The function $u$ is determined by the field $f^0$ uniquely up to an additive constant.

The definition of the operator $d$ gives with the help of (\ref{4.3})
$$
(df^0)_{11}=-{\nabla}_{\!1}{\nabla}_{\!2}u,\quad
(df^0)_{12}=\frac{1}{2}({\nabla}_{\!1}{\nabla}_{\!1}u-{\nabla}_{\!1}{\nabla}_{\!1}u),\quad (df^0)_{22}={\nabla}_{\!1}{\nabla}_{\!2}u.
$$
Observe $df^0$ turns out to be a trace free field: $j(df^0)=e^{-2\mu}\big((df^0)_{11}+(df^0)_{22}\big)=0$. Therefore $pdf^0=df^0$. Thus, the second of equations (\ref{4.2}) is written in isothermal coordinates as the system
$$
{\nabla}_{\!1}{\nabla}_{\!2}u=Z^{2,c}_{11},\quad
\frac{1}{2}\big({\nabla}_{\!1}{\nabla}_{\!1}u-{\nabla}_{\!2}{\nabla}_{\!2}u\big)=-Z^{2,c}_{12}.
$$
Comparing (\ref{3.8}) and (\ref{4.0}), we see that $Z^{2,c}_{11}=T^c_{12},\ Z^{2,c}_{12}=-T^c_{11}$. Therefore the previous system can be rewritten in the form
\begin{equation}
\frac{1}{2}\big({\nabla}_{\!1}{\nabla}_{\!1}u-{\nabla}_{\!2}{\nabla}_{\!2}u\big)=T^c_{11},\quad
{\nabla}_{\!1}{\nabla}_{\!2}u=T^c_{12}.
                         \label{4.4}
\end{equation}

Let us write system (\ref{4.4}) in an invariant form. To this end first of all we observe that the Hessian ${\nabla}{\nabla}u=({\nabla}_{\!i}{\nabla}_{\!j}u)$ of the function $u$ is a well defined symmetric tensor field on the torus since partial derivative of the function are $\Gamma$-periodic. The Riemannian Laplacian $\Delta u=\mbox{tr}({\nabla}{\nabla}u)=g^{ij}{\nabla}_{\!i}{\nabla}_{\!j}u$ is a well defined function on the torus. Let us now consider the trace free part of the Hessian
$$
{\nabla}{\nabla}u-\frac{1}{2}(\Delta u)g,
$$
where $g$ is the metric tensor. In isothermal coordinates
$$
\begin{aligned}
&\big({\nabla}{\nabla}u-\frac{1}{2}(\Delta u)g\big)_{11}=-\big({\nabla}{\nabla}u-\frac{1}{2}(\Delta u)g\big)_{22} =\frac{1}{2}\big({\nabla}_{\!1}{\nabla}_{\!1}u-{\nabla}_{\!2}{\nabla}_{\!2}u\big),\\
&\big({\nabla}{\nabla}u-\frac{1}{2}(\Delta u)g\big)_{12}={\nabla}_{\!1}{\nabla}_{\!2}u.
\end{aligned}
$$
Comparing these equalities with (\ref{4.4}), we see that system (\ref{4.4}) is equivalent to the equation
\begin{equation}
{\nabla}{\nabla}u-\frac{1}{2}(\Delta u)g=T^c.
                         \label{4.5}
\end{equation}

Thus, the question: ``Does there exist a Riemannian metric on the 2-torus which admits a rank 3 irreducible Killing tensor field?'' is closely related to the solvability problem  for equation (\ref{4.5}): one has to find necessary and sufficient conditions on coefficients and right-hand side of the equation (i.e. conditions on $(\Gamma,\mu,c)$) for the existence of a solution with $\Gamma$-periodic partial derivatives. The problem seems to be rather hard.
In the current section, we will obtain two necessary solvability condition for (\ref{4.5}) which are of some interest.

Let us demonstrate equation (\ref{4.5}) can be solved with respect to all third order derivatives of the function $u$. Now, we perform our calculations in arbitrary coordinates. We introduce the temporary notation $v=\frac{1}{2}\Delta u$. Differentiate (\ref{4.5}) to obtain
\begin{equation}
{\nabla}_{\!i}{\nabla}_{\!j}{\nabla}_{\!k}u=g_{jk}{\nabla}_{\!i}v+{\nabla}_{\!i}T^c_{jk}.
                         \label{4.6}
\end{equation}
By the commutator formula for covariant derivatives,
$$
\begin{aligned}
&{\nabla}_{\!1}{\nabla}_{\!1}{\nabla}_{\!2}u-{\nabla}_{\!2}{\nabla}_{\!1}{\nabla}_{\!1}u=-R^1_{\ 112}{\nabla}_{\!1}u-R^2_{\ 112}{\nabla}_{\!2}u,\\
&{\nabla}_{\!1}{\nabla}_{\!2}{\nabla}_{\!2}u-{\nabla}_{\!2}{\nabla}_{\!1}{\nabla}_{\!2}u=-R^1_{\ 212}{\nabla}_{\!1}u-R^2_{\ 212}{\nabla}_{\!2}u,
\end{aligned}
$$
where $R=(R^i_{\ jk\ell})$ is the curvature tensor.
Substituting values (\ref{4.6}) for third order derivatives into left-hand sides of these equalities, we arrive to the system
$$
\begin{aligned}
&g_{12}{\nabla}_{\!1}v-g_{11}{\nabla}_{\!2}v=-R^1_{\ 112}{\nabla}_{\!1}u-R^2_{\ 112}{\nabla}_{\!2}u+{\nabla}_{\!2}T^c_{11}-{\nabla}_{\!1}T^c_{12},\\
&g_{22}{\nabla}_{\!1}v-g_{12}{\nabla}_{\!2}v=-R^1_{\ 212}{\nabla}_{\!1}u-R^2_{\ 212}{\nabla}_{\!2}u+{\nabla}_{\!2}T^c_{12}-{\nabla}_{\!1}T^c_{22}.
\end{aligned}
$$
We solve the system and get
$$
\begin{aligned}
&{\nabla}_{\!1}v=-R^{12}_{\ \ 12}{\nabla}_{\!1}u+{\nabla}_{\!2}(T^c)^2_{1}-{\nabla}_{\!1}(T^c)^2_{2},\\
&{\nabla}_{\!2}v=-R^{12}_{\ \ 12}{\nabla}_{\!2}u-{\nabla}_{\!2}(T^c)^1_{1}+{\nabla}_{\!1}(T^c)^1_{2},
\end{aligned}
$$
where $(T^c)^i_j=g^{ip}T^c_{pj}$.
The condition $\mbox{tr}\,T^c=g^{ij}T^c_{ij}=0$ implies
$$
{\nabla}_{\!1}(T^c)^2_{2}=-{\nabla}_{\!1}(T^c)^1_{1}, \quad {\nabla}_{\!2}(T^c)^1_{1}=-{\nabla}_{\!2}(T^c)^2_{2}
$$
and two previous formulas can be written uniformly:
$$
{\nabla}_{\!i}v=-R^{12}_{\ \ 12}{\nabla}_{\!i}u+{\nabla}^pT^c_{ip},
$$
where ${\nabla}^p=g^{pq}{\nabla}_{\!q}$. Recall the formula $R_{ijk\ell}=K(g_{ik}g_{j\ell}-g_{i\ell}g_{jk})$ holds in the two-dimensional case where $K$ is the Gaussian curvature. Hence
$R^{12}_{\ \ 12}=K$ and the previous formula takes the form
$$
{\nabla}_{\!i}v=-K{\nabla}_{\!i}u+{\nabla}^pT^c_{ip}.
$$
Substituting this value into (\ref{4.6}), we obtain the final formula
\begin{equation}
{\nabla}_{\!i}{\nabla}_{\!j}{\nabla}_{\!k}u=g_{jk}(-K{\nabla}_{\!i}u+{\nabla}^pT^c_{ip})+{\nabla}_{\!i}T^c_{jk}.
                         \label{4.7}
\end{equation}

Now, we are going to derive some solvability condition for system (\ref{4.7}). In the case when third order partial derivatives stand on left-hand sides, the standard approach for deriving such solvability conditions consists of differentiating equations and using the symmetry of fourth order partial derivatives. In our case, covariant derivatives stand on left-hand sides. Therefore we need to use corresponding commutator formulas.

Differentiate (\ref{4.7}) to obtain
$$
{\nabla}_{\!i}{\nabla}_{\!j}{\nabla}_{\!k}{\nabla}_{\!\ell}u=g_{k\ell}(-K{\nabla}_{\!i}{\nabla}_{\!j}u-{\nabla}_{\!i}K\cdot{\nabla}_{\!j}u+{\nabla}_{\!i}{\nabla}^pT^c_{jp})
+{\nabla}_{\!i}{\nabla}_{\!j}T^c_{k\ell}.
$$
Then we alternate this equality in the indices $(i,j)$ and write the result as follows:
\begin{equation}
\begin{aligned}
g_{k\ell}(-{\nabla}_{\!j}K\cdot{\nabla}_{\!i}u&+{\nabla}_{\!i}K\cdot{\nabla}_{\!j}u
-{\nabla}_{\!i}{\nabla}^pT^c_{jp}+{\nabla}_{\!j}{\nabla}^pT^c_{ip})\\
&=({\nabla}_{\!i}{\nabla}_{\!j}T^c_{k\ell}-{\nabla}_{\!j}{\nabla}_{\!i}T^c_{k\ell})
-({\nabla}_{\!i}{\nabla}_{\!j}{\nabla}_{\!k}{\nabla}_{\!\ell}u-{\nabla}_{\!j}{\nabla}_{\!i}{\nabla}_{\!k}{\nabla}_{\!\ell}u).
\end{aligned}
                         \label{4.8}
\end{equation}
Let us demonstrate the right-hand side of this formula is identically equal to zero. Indeed, by the commutator formula for covariant derivatives,
\begin{equation}
{\nabla}_{\!i}{\nabla}_{\!j}T^c_{k\ell}-{\nabla}_{\!j}{\nabla}_{\!i}T^c_{k\ell}
=-R^p_{\ kij}T^c_{p\ell}-R^p_{\ \ell ij}T^c_{kp},
                         \label{4.9}
\end{equation}
$$
{\nabla}_{\!i}{\nabla}_{\!j}{\nabla}_{\!k}{\nabla}_{\!\ell}u-{\nabla}_{\!j}{\nabla}_{\!i}{\nabla}_{\!k}{\nabla}_{\!\ell}u
=-R^p_{\ kij}{\nabla}_{\!p}{\nabla}_{\!\ell}u-R^p_{\ \ell ij}{\nabla}_{\!k}{\nabla}_{\!p}u.
$$
Substituting values (\ref{4.5}) for second order derivatives of the function $u$ into the right-hand side of the last formula, we obtain
$$
{\nabla}_{\!i}{\nabla}_{\!j}{\nabla}_{\!k}{\nabla}_{\!\ell}u-{\nabla}_{\!j}{\nabla}_{\!i}{\nabla}_{\!k}{\nabla}_{\!\ell}u
=-R^p_{\ kij}T^c_{p\ell}-R^p_{\ \ell ij}T^c_{kp}-\frac{1}{2}(\Delta u)R_{\ell kij}-\frac{1}{2}(\Delta u)R_{k\ell ij}.
$$
The sum of two last terms on the right-hand side is equal to zero in view of symmetries of the curvature tensor and the formula is simplified to the following one:
$$
{\nabla}_{\!i}{\nabla}_{\!j}{\nabla}_{\!k}{\nabla}_{\!\ell}u-{\nabla}_{\!j}{\nabla}_{\!i}{\nabla}_{\!k}{\nabla}_{\!\ell}u
=-R^p_{\ kij}T^c_{p\ell}-R^p_{\ \ell ij}T^c_{kp}.
$$
From this and (\ref{4.9}), we see that the right-hand side of (\ref{4.8}) is indeed equal to zero. Now, (\ref{4.8}) takes the form
$$
-{\nabla}_{\!j}K\cdot{\nabla}_{\!i}u+{\nabla}_{\!i}K\cdot{\nabla}_{\!j}u
={\nabla}_{\!i}{\nabla}^pT^c_{jp}-{\nabla}_{\!j}{\nabla}^pT^c_{ip}.
$$
Setting $(i,j)=(1,2)$ here, we arrive to the equality
$$
-{\nabla}_{\!2}K\cdot{\nabla}_{\!1}u+{\nabla}_{\!1}K\cdot{\nabla}_{\!2}u
={\nabla}_{\!1}{\nabla}^pT^c_{2p}-{\nabla}_{\!2}{\nabla}^pT^c_{1p}
$$
that can be written in the form
\begin{equation}
\big(-{\nabla}_{\!2}K\frac{\partial}{\partial x^1}+{\nabla}_{\!1}K\frac{\partial}{\partial x^2}\Big)u
={\nabla}_{\!1}{\nabla}^pT^c_{2p}-{\nabla}_{\!2}{\nabla}^pT^c_{1p}.
                         \label{4.10}
\end{equation}

We assume the torus to be oriented and a coordinate system to be agreed with the orientation so that the shortest rotation from $\partial/\partial x^1$ to $\partial/\partial x^2$ goes in the positive direction. Let ${\nabla}^\bot K$ be the vector field obtained from the gradient ${\nabla}K$ by rotating through the right angle in the positive direction. Then
$$
{\nabla}^\bot K=(g_{11}g_{22}-g_{12}^2)^{-1/2}\Big(-{\nabla}_{\!2}K\frac{\partial}{\partial x^1}+{\nabla}_{\!1}K\frac{\partial}{\partial x^2}\Big)
$$
and equation (\ref{4.10}) takes the final form
\begin{equation}
({\nabla}^\bot K)u=\Phi^c,
                         \label{4.11}
\end{equation}
where
\begin{equation}
\Phi^c=(g_{11}g_{22}-g_{12}^2)^{-1/2}\big({\nabla}_{\!1}{\nabla}^pT^c_{2p}-{\nabla}_{\!2}{\nabla}^pT^c_{1p}\big).
                         \label{4.12}
\end{equation}
In particular, (\ref{4.11}) implies that $\Phi^c$ is a well defined smooth function on the torus. The latter fact can be proved directly by checking that the right-hand side of (\ref{4.12}) is independent of the choice of coordinates.

Substituting values (\ref{3.8}) for components of the tensor $T^c$ into (\ref{4.12}) and performing some easy calculations, we obtain the following expression for the function $\Phi^c$ in global isothermal coordinates:
\begin{equation}
\Phi^c=c^1\Lambda_1+c^2\Lambda_2,
                         \label{4.13}
\end{equation}
where
\begin{equation}
\begin{aligned}
\Lambda_1&=\mu_{xxx}-3\mu_{xyy}+10\mu_x\mu_{xx}-20\mu_y\mu_{xy}-10\mu_x\mu_{yy}+8\mu_x^3-24\mu_x\mu_y^2,\\
\Lambda_2&=3\mu_{xxy}-\mu_{yyy}+10\mu_y\mu_{xx}+20\mu_x\mu_{xy}-10\mu_y\mu_{yy}+24\mu_x^2\mu_y-8\mu_y^3.
\end{aligned}
                         \label{4.14}
\end{equation}

As we know, $c=(c^1,c^2)$ is a pseudovector of weight 3. Equality (\ref{4.13}) gives us an impetus to the suggestion: $\Lambda=(\Lambda_1,\Lambda_2)$ must be a 1-pseudoform of weight 3. This fact is not obvious from (\ref{4.14}). To clarify the situation, let us find the complex version of formulas (\ref{4.14}). Using the equalities
$$
e^{2\mu}=\lambda,\quad \partial_x=\partial_z+\partial_{\bar z},\quad \partial_y=\textsl{i}(\partial_z-\partial_{\bar z})
$$
and performing some easy calculations, we transform (\ref{4.14}) to the form
\begin{equation}
\begin{aligned}
\Lambda_1&=2\frac{\lambda_{zzz}+\lambda_{\bar z\bar z\bar z}}{\lambda}
+4\frac{\lambda_z\lambda_{zz}+\lambda_{\bar z}\lambda_{\bar z\bar z}}{\lambda^2}
-2\frac{\lambda_z^3+\lambda_{\bar z}^3}{\lambda^3},\\
\Lambda_2&=\textsl{i}\Big(2\frac{\lambda_{zzz}-\lambda_{\bar z\bar z\bar z}}{\lambda}
+4\frac{\lambda_z\lambda_{zz}-\lambda_{\bar z}\lambda_{\bar z\bar z}}{\lambda^2}
-2\frac{\lambda_z^3-\lambda_{\bar z}^3}{\lambda^3}\Big).
\end{aligned}
                         \label{4.15}
\end{equation}
An invariant nature of these formulas is now obvious which is expressed in our language by the statement: $\Lambda=(\Lambda_1,\Lambda_2)$ is a 1-pseudoform of weight 3.
The summands on right-hand sides of (\ref{4.15}) give us three examples of 1-pseudoforms of weight 3:
\begin{equation}
\begin{aligned}
\Lambda^{(1)}&=(\Lambda^{(1)}_1,\Lambda^{(1)}_2)=\frac{1}{\lambda}\big(\lambda_{zzz}+\lambda_{\bar z\bar z\bar z}, \textsl{i}(\lambda_{zzz}-\lambda_{\bar z\bar z\bar z})\big),\\
\Lambda^{(2)}&=(\Lambda^{(2)}_1,\Lambda^{(2)}_2)=\frac{1}{\lambda^2}\big(\lambda_z\lambda_{zz}+\lambda_{\bar z}\lambda_{\bar z\bar z},
 \textsl{i}(\lambda_z\lambda_{zz}-\lambda_{\bar z}\lambda_{\bar z\bar z})\big),\\
\Lambda^{(3)}&=(\Lambda^{(3)}_1,\Lambda^{(3)}_2)=\frac{1}{\lambda^3}\big(\lambda_z^3+\lambda_{\bar z}^3, \textsl{i}(\lambda_z^3-\lambda_{\bar z}^3)\big).
\end{aligned}
                         \label{4.16}
\end{equation}
We emphasize there is no ambiguity in the definition, i.e.,  these 1-pseudoforms are completely determined by the metric $g$ as well as the 1-pseudoform of weight 3 participating in (\ref{4.13})
\begin{equation}
\Lambda=2\Lambda^{(1)}+4\Lambda^{(2)}-2\Lambda^{(3)}.
                         \label{4.17}
\end{equation}

The expression on the left-hand side of (\ref{4.11}) is the derivative of the function $u$ along the isoline $\gamma$ of the function $K$ which is parameterized so that $\|\dot\gamma\|=\|{\nabla}K\|$. Integrating (\ref{4.11}) over $\gamma$, we arrive to the following statement.

\begin{theorem} \label{Th4.1}
Let $({\mathbb T}^2,g)=({\mathbb R}^2/\Gamma,\lambda)$ be a two-dimensional Riemannian torus. If the torus admits a real irreducible rank 3 Killing tensor field, then the 1-pseudoform $\Lambda$ of weight 3, which is defined by (\ref{4.16})--(\ref{4.17}), satisfies the following condition.

There exist a real constant pseudovector $c\neq 0$ of weight 3 and real function $u\in C^\infty({\mathbb R}^2)$ with $\Gamma$-periodic partial derivatives such that the following statement holds.

Let a curve $\gamma:[a,b]\rightarrow {\mathbb T}^2$ be a part of an isoline $\{K=K_0\}$ of the Gaussian curvature $K$. Assume $\gamma$ do not contain critical points of the function $K$ and to be parameterized so that $\|\dot\gamma\|=\|{\nabla}K\|$. Let $\tilde\gamma:[a,b]\rightarrow {\mathbb R}^2$ be the lift of $\gamma$ with respect to the covering
${\mathbb R}^2\rightarrow{\mathbb R}^2/\Gamma={\mathbb T}^2$. Then
\begin{equation}
\int\limits_a^b\big(c^1\Lambda_1(\gamma(t))+c^2\Lambda_2(\gamma(t)\big)\,dt=u(\tilde\gamma(b))-u(\tilde\gamma(a)).
                         \label{4.18}
\end{equation}
\end{theorem}

Since the function $u\in C^\infty({\mathbb R}^2)$ has $\Gamma$-periodic derivatives $u_x$ and $u_y$, it can be uniquely represented in the form
\begin{equation}
u(x,y)=w(x,y)+\alpha_1x+\alpha_2y,
                         \label{4.19}
\end{equation}
where $\alpha_1,\alpha_2\in{\mathbb R}$ and $w$ is a $\Gamma$-periodic function. The expression $\alpha=\alpha_1\,dx+\alpha_2\,dy$ is a well defined closed 1-form on the torus independent of the choice of global isothermal coordinates. Let $\sigma=[\alpha]\in H^1({\mathbb T}^2,{\mathbb R})$ be the one-dimensional cohomology class defined by the form $\alpha$.

Now, we consider the case of a closed curve $\gamma:[a,b]\rightarrow{\mathbb T}^2$ participating in Theorem \ref{Th4.1}. If $\tilde\gamma(t)=\big(\tilde\gamma{}^1(t),\tilde\gamma{}^2(t)\big)$ is the lift of $\gamma$, then the vector $\tilde\gamma(b)-\tilde\gamma(a)$ belongs to the lattice $\Gamma$. Using representation (\ref{4.19}), we write the right-hand side of (\ref{4.18}) in the form
$$
u(\tilde\gamma(b))-u(\tilde\gamma(a))=\Big[w(\tilde\gamma(b))-w(\tilde\gamma(a))\Big]
+\Big[\alpha_1(\tilde\gamma{}^1(b)-\tilde\gamma{}^1(a))+\alpha_2(\tilde\gamma{}^2(b)-\tilde\gamma{}^2(a))\Big].
$$
The difference in the first brackets is equal to zero since $w$ is a $\Gamma$-periodic function. The expression in the second brackets is obviously equal to $\langle\sigma,[\gamma]\rangle$, where $[\gamma]\in H_1({\mathbb T}^2,{\mathbb R})$ is the one-dimensional homology class determined by the cycle $\gamma$ and
$$
\langle\cdot,\cdot\rangle:H^1({\mathbb T}^2,{\mathbb R})\times H_1({\mathbb T}^2,{\mathbb R})\rightarrow{\mathbb R}
$$
is the canonical paring of one-dimensional de Rham cohomologies and homologies. In this way we arrive to the following statement.

\begin{theorem} \label{Th4.2}
Let $({\mathbb T}^2,g)$ be a two-dimensional Riemannian torus. If the torus admits a real irreducible rank 3 Killing tensor field, then the 1-pseudoform $\Lambda$ of weight 3, which is defined by (\ref{4.16})--(\ref{4.17}), satisfies the following condition.

There exist a real constant pseudovector $c\neq 0$ of weight 3 and cohomology class $\sigma\in H^1({\mathbb T},{\mathbb R})$ such that the equality
\begin{equation}
\oint\limits_\gamma(c^1\Lambda_1+c^2\Lambda_2)\,dt=\langle\sigma,[\gamma]\rangle
                         \label{4.20}
\end{equation}
holds for every closed curve $\gamma:[a,b]\rightarrow {\mathbb T}^2$ which is a part of an isoline $\{K=K_0\}$, does not contain critical points of $K$, and parameterized so that $\|\dot\gamma\|=\|{\nabla}K\|$.
\end{theorem}

Unlike Theorem \ref{Th4.1}, the function $u$ is not mentioned here. Therefore the statement of Theorem \ref{Th4.2} can be considered as a necessary condition for the solvability of equation (\ref{4.5}).

\begin{corollary} \label{C4.1}
Under hypotheses of Theorem \ref{Th4.2}, let $\gamma$ be a contractible closed curve that is a part of an isoline $\{K=K_0\}$ and let $D$ be the closed domain on the torus which is homeomorphic to a disk and is bounded by $\gamma$. Assume there is exactly one critical point of the function $K$ in $D$ and moreover the point belongs to the interior of $D$ and is a nondegenerate critical point either of index 0 or of index 2 (i.e., it is a point either of a local maximum or of a local minimum). Then
\begin{equation}
\int\limits_D(c^1\Lambda_1+c^2\Lambda_2)\,d\sigma=0,
                         \label{4.21}
\end{equation}
where $d\sigma$ is the area form.
\end{corollary}

\begin{corollary} \label{C4.2}
Under hypotheses of Theorem \ref{Th4.2}, let $D$ be an annulus domain on the torus homeomorphic to the product of a segment and circle. Assume $D$ do not contain critical points of the function $K$ and assume both boundary circles to be parts of isolines $\{K=K_0\}$ and $\{K=K_1\}$ respectively. Then
\begin{equation}
\int\limits_D(c^1\Lambda_1+c^2\Lambda_2)\,d\sigma=\pm\langle\sigma,[\gamma]\rangle(K_1-K_0),
                         \label{4.22}
\end{equation}
where $\gamma$ is one of boundary circles of the domain $D$.
\end{corollary}

To prove Corollaries \ref{C4.1} and \ref{C4.2}, it suffices to observe that, if the parametrization $\gamma(t)$ of an isoline is chosen as indicated in Theorem \ref{Th4.2}, then $dt\wedge dK=\pm d\sigma$.

\bigskip

In conclusion, let us return to (\ref{4.5}) and derive some new fourth order equation for the function $u$. Obviously, $d^2u$ is the Hessian of the function $u$ and $pd^2u$ is the trace-free part of the Hessian. Hence (\ref{4.5}) can be written in the form
\begin{equation}
pd^2u=T^c.
                         \label{4.23}
\end{equation}
The operator  $\delta^2$ is adjoint to $pd^2$. Apply $\delta^2$ to both sides of (\ref{4.23}) to obtain
\begin{equation}
\delta^2pd^2u=\delta^2T^c.
                         \label{4.24}
\end{equation}
Observe the passage from (\ref{4.23}) to (\ref{4.24}) is not reversible. Therefore our further conclusions should be considered as necessary conditions for the solvability of equation (\ref{4.5}).

The fourth order operator $\delta^2pd^2$ is quite similar to the second power of the Laplacian. To clarify the similarity, we perform some calculations in coordinates. First of all,
$$
(pd^2u)_{ij}={\nabla}_{\!i}{\nabla}_{\!j}u-\frac{1}{2}g_{ij}\Delta u.
$$
Hence
$$
\delta^2pd^2u={\nabla}^i\nabla^j\big({\nabla}_{\!i}{\nabla}_{\!j}u-\frac{1}{2}g_{ij}\Delta u\big)
={\nabla}^i\nabla^j{\nabla}_{\!i}{\nabla}_{\!j}u-\frac{1}{2}\Delta^2 u.
$$
Permuting the derivatives $\nabla^j$ and ${\nabla}_{\!i}$ in the first term on the right-hand side with the help of the corresponding commutator formula, we obtain
$$
\delta^2pd^2u=\frac{1}{2}\Delta^2 u-{\nabla}^i\big(R^j_i{\nabla}_{\!j}u),
$$
where $R^j_i$ is the Ricci tensor. In the two-dimensional case, $R^j_i=-K\delta^j_i$, where $K$ is the Gaussian curvature, and the last formula takes the form
$$
\delta^2pd^2u=\frac{1}{2}\Delta^2 u+{\nabla}^i\big(K{\nabla}_{\!i}u).
$$
Equation (\ref{4.24}) is thus equivalent to the following one:
\begin{equation}
\frac{1}{2}\Delta^2 u+\delta(Kdu)=\delta^2T^c.
                         \label{4.25}
\end{equation}

Now, we evaluate the right-hand side of (\ref{4.25}). We have already calculated the divergence of an arbitrary trace-free tensor field, formula (\ref{3.7'}). Applying this formula to $T^c$ and using (\ref{4.0}), we find in global isothermal coordinates
$$
\begin{aligned}
(\delta T^c)_{1}&=e^{2\mu}(-c^2\mu_{xx}+2c^1\mu_{xy}+c^2\mu_{yy}-4c^2\mu_{x}^2+8c^1\mu_{x}\mu_{y}+4c^2\mu_{y}^2),\\
(\delta T^c)_{2}&=e^{2\mu}(c^1\mu_{xx}+2c^2\mu_{xy}-c^1\mu_{yy}+4c^1\mu_{x}^2+8c^2\mu_{x}\mu_{y}-4c^1\mu_{y}^2).
\end{aligned}
$$
Substituting these values into the formula
$$
\delta^2T^c=e^{-2\mu}\Big(\frac{\partial(\delta T^c)_1}{\partial x}+\frac{\partial(\delta T^c)_2}{\partial y}\Big),
$$
we obtain
$$
\delta^2T^c=-c^2\Lambda_1+c^1\Lambda_2,
$$
where $\Lambda=(\Lambda_1,\Lambda_2)$ is defined by (\ref{4.14}). Thus, equation (\ref{4.25}) takes the form
$$
\frac{1}{2}\Delta^2 u+\delta(Kdu)=-c^2\Lambda_1+c^1\Lambda_2,
$$
or in more traditional notations,
\begin{equation}
\frac{1}{2}\Delta^2 u+\mbox{\rm div}(K{\nabla}u)=-c^2\Lambda_1+c^1\Lambda_2.
                         \label{4.26}
\end{equation}
Uniting equations (\ref{4.11}) and (\ref{4.26}), we arrive to the following statement.

\begin{theorem} \label{Th4.3}
If a Riemannian torus $({\mathbb R}^2/\Gamma,g)$ admits a rank 3 irreducible Killing tensor field, then there exists a constant real pseudovector $0\neq c=(c^1,c^2)$ such that the system of equations
$$
({\nabla}^\bot K)u=c^1\Lambda_1+c^2\Lambda_2,\quad \frac{1}{2}\Delta^2 u+\mbox{\rm div}(K{\nabla}u)=-c^2\Lambda_1+c^1\Lambda_2
$$
has a solution
$u\in C^\infty({\mathbb R}^2)$ with $\Gamma$-periodic derivatives $u_x$ and $u_y$. Here $K$ is the Gaussian curvature and $\Lambda=(\Lambda_1,\Lambda_2)$ is the 1-pseudoform of weight 3 which is defined by (\ref{4.16})--(\ref{4.17}) in global isothermal coordinates.
\end{theorem}

\begin{corollary} \label{C4.3}
If a Riemannian torus $({\mathbb T}^2,g)$ admits a rank 3 irreducible Killing tensor field, then both the components of the 1-pseudoform $\Lambda=(\Lambda_1,\Lambda_2)$ have zero mean values, i.e.,
\begin{equation}
\int\limits_{{\mathbb T}^2}\Lambda_1\,d\sigma=0,\quad \int\limits_{{\mathbb T}^2}\Lambda_2\,d\sigma=0,
                         \label{4.27}
\end{equation}
where $d\sigma=e^{2\mu}\,dxdy$ is the area form.
\end{corollary}

\begin{proof}
The left-hand side of (\ref{4.26}) integrates to zero over the torus. Hence
\begin{equation}
-c^2\int\limits_{{\mathbb T}^2}\Lambda_1\,d\sigma+c^1\int\limits_{{\mathbb T}^2}\Lambda_2\,d\sigma=0.
                         \label{4.28}
\end{equation}

By (\ref{4.12})--(\ref{4.13}), in global isothermal coordinates,
$$
c^1\Lambda_1+c^2\Lambda_2=e^{-2\mu}\big({\nabla}_{\!1}(\delta T^c)_2-{\nabla}_{\!2}(\delta T^c)_1\big)={\nabla}^1(\delta T^c)_2-{\nabla}^2(\delta T^c)_1.
$$
Introduce the covector field $v=(\delta T^c)^\bot=-(\delta T^c)_2\,dx+(\delta T^c)_1\,dy$. The previous formula can be rewritten in terms of $v$ as follows:
$$
c^1\Lambda_1+c^2\Lambda_2=-({\nabla}^1v_1+{\nabla}^2v_2)=-\delta v.
$$
Hence
\begin{equation}
c^1\int\limits_{{\mathbb T}^2}\Lambda_1\,d\sigma+c^2\int\limits_{{\mathbb T}^2}\Lambda_2\,d\sigma=0.
                         \label{4.29}
\end{equation}
Equalities (\ref{4.28}) and (\ref{4.29}) imply (\ref{4.27}) since $c\neq0$.
\end{proof}


\end{document}